\documentclass[11pt]{amsart} 
\usepackage{amsmath,amssymb,bm, color}

\usepackage{amsmath}
\usepackage{graphicx}

\usepackage{mathrsfs}
\usepackage{bbm}
\usepackage{bm}
\usepackage{amsfonts,amssymb}
\usepackage{multirow}
\usepackage{lineno}
\usepackage{color}

\numberwithin{equation}{section}
 \newtheorem{theorem}{Theorem}[section]
 \newtheorem{lemma}[theorem]{Lemma}

\def\3bar{{|\hspace{-.02in}|\hspace{-.02in}|}}
\def\E{{\mathcal{E}}}
\def\T{{\mathcal{T}}}

\def\cal#1{{\mathcal #1}}

\def\bn{{\mathbf{n}}}

\def\ljump{{[\![}}
\def\rjump{{]\!]}}

\newtheorem{algorithm}{Weak Galerkin Algorithm}[section]

\numberwithin{equation}{section}

\def\3bar{{|\hspace{-.02in}|\hspace{-.02in}|}}

 \def\cal#1{\mathcal{#1}}
\def\ad#1{\begin{aligned}#1\end{aligned}}   
\def\a#1{\begin{align*}#1\end{align*}} \def\an#1{\begin{align}#1\end{align}}

\begin{document}

\title[WG finite element]
 {Weak Galerkin Finite Element Methods for Optimal Control Problems Governed by   Second Order Elliptic Partial Differential Equations}

\author {Chunmei Wang}
\address{Department of Mathematics, University of Florida, Gainesville, FL 32611, USA. }
\email{chunmei.wang@ufl.edu}
\thanks{The research of Chunmei Wang was partially supported by National Science Foundation Grants DMS-2136380 and DMS-2206332.}

\author {Junping Wang}
\address{Division of Mathematical Sciences, National Science Foundation, Alexandria, VA 22314, USA. }
\email{jwang@nsf.gov} 
\thanks{The research of Junping Wang was supported by the NSF IR/D program, while working at National Science Foundation. However, any opinion, finding, and conclusions or recommendations
expressed in this material are those of the author and do not
necessarily reflect the views of the National Science Foundation.}

\author { Shangyou Zhang }
\address{Department of Mathematical
            Sciences, University
     of Delaware, Newark, DE 19716, USA. }
\email{szhang@udel.edu }

\begin{abstract}

This paper is concerned with the development of weak Galerkin (WG) finite element method for optimal control problems governed by   second order elliptic partial differential equations (PDEs). It is advantageous to use discontinuous finite elements over the traditional $C^1$
  finite elements here.  Optimal order error estimates are established and confirmed by
  some numerical tests.
  
\end{abstract}


 
\maketitle

\section{Introduction}

Let $\Omega\subset \mathbb R^2$ be an open bounded domain with Lipschitz boundary $\partial\Omega$ which is decomposed into a homogeneous Neumann part $\Gamma_N$ and a control part $\Gamma_C$ on which the control acts ($\partial\Omega=\Gamma_C\cup\Gamma_N$),
\begin{equation}\label{model}
    \begin{split}
        -\Delta u+u&=f, \qquad \text{in} \ \Omega,\\
        \partial_n u&=0, \qquad\text{on}
    \ \Gamma_N, \\
        \partial_n u&=q, \qquad\text{on}
    \ \Gamma_C.
    \end{split}
\end{equation}
The observations are given on a part $\Gamma_O$ of the boundary and the associated cost functional is 
\begin{equation}\label{model2}
 \frac{1}{2}\|u-c_0\|^2_{\Gamma_O}+\frac{\alpha}{2}\| q\|^2_{\Gamma_C},   
\end{equation}
with a regularization parameter $\alpha>0$. 

 
The model problem \eqref{model}-\eqref{model2} can be reformulated to the following constrained minimization problem; i.e.,
\begin{equation}\label{mini}
\min \left(\frac{1}{2}\|u-c_0\|^2_{\Gamma_O}+\frac{\alpha}{2}\| \partial_n u\|^2_{\Gamma_C}\right),
\end{equation}  
subject to 
\begin{equation}\label{constr}
    \begin{split}
        -\Delta u+u&=f, \qquad \text{in} \ \Omega,\\
        \partial_n u&=0, \qquad\text{on}
    \ \Gamma_N.
    \end{split}
\end{equation}

Optimal control problems governed by second-order elliptic partial differential equations (PDEs) have numerous applications across various scientific and engineering domains. In engineering disciplines such as structural engineering, optimal control problems can be used to design structures that meet specific criteria for strength, stability, and safety. For example, optimizing the shape of an aircraft wing to minimize drag while maintaining structural integrity. In fluid dynamics, optimal control problems can be applied to optimize fluid flows in scenarios such as designing efficient airfoil shapes, minimizing turbulence, or optimizing chemical reactions in flow systems. In medical imaging, optimizing the parameters of imaging devices or processes can improve the quality of images while minimizing exposure to radiation. In environmental science, Optimal control can be used to model and manage environmental processes, such as groundwater contamination remediation, where the goal is to optimize the distribution of control agents to minimize the spread of pollutants. In material engineering, optimizing the properties of materials (e.g., electrical conductivity, mechanical strength) subject to physical constraints can have applications in developing new materials for specific applications. The common thread in optimal control is to find the best control strategies to achieve specific goals while considering the underlying dynamics governed by these equations. 

There is a growing literature (see an incomplete list \cite{boppss, bs,bsg, bsz, bsz2, bsz3, cmv, dh, gt, wy, hpuu, lgy, meyer, npr}) on finite element methods for the optimal control problems governed by partial differential equations. The weak Galerkin (WG) method, first introduced for second-order elliptic problems \cite{wy}, provides a natural extension of the classical finite element method through a relaxed regularity of the approximating functions. This novelty provides a high flexibility in numerical approximations with any needed accuracy and 
mesh generation being general polygonal or polyhedral partitions.  
To the best of our knowledge, there are not any results on weak Galerkin for optimal control problem governed by PDEs. In this paper, we shall develop a weak Galerkin method for optimal control problems governed by second order elliptic problems. We shall  prove the optimal order of error estimates in the energy norm and the $L^2$ norm. Numerical results will verify the established theoretical results.

This paper is structured as follows. Section 2 provides a weak formulation of the model problem \eqref{model}-\eqref{model2}.
In Section 3, we briefly review of the definition of the weak Laplacian and its discrete version. In Section 4, we present the weak Galerkin scheme and derive the existence and uniqueness of the solution. Stability analysis is provided in Section 5. Error equations for the proposed weak Galerkin scheme are derived in Section 6. Section 7 focuses on deriving the error estimate for the numerical approximation in the energy norm. Section 8 is devoted to establishing the error estimate for the numerical approximation in the $L^2$ norm. Finally, in Section  9, we present a series of numerical results to validate the theoretical developments presented in the previous sections.

The standard notations are adopted throughout this paper. Let $D$ be any open bounded domain with Lipschitz continuous boundary in $\mathbb{R}^d$. We use $(\cdot,\cdot)_{s,D}$, $|\cdot|_{s,D}$ and $\|\cdot\|_{s,D}$ to denote the inner product, semi-norm and norm in the Sobolev space $H^s(D)$ for any integer $s\geq0$, respectively. For simplicity, the subscript $D$ is  dropped from the notations of the inner product and norm when the domain $D$ is chosen as $D=\Omega$. For the case of $s=0$, the notations $(\cdot,\cdot)_{0,D}$, $|\cdot|_{0,D}$ and $\|\cdot\|_{0,D}$ are simplified as $(\cdot,\cdot)_D$, $|\cdot|_D$ and $\|\cdot\|_D$, respectively. The notation ``$A\lesssim B$'' refers to the inequality ``$A\leq CB$'' where $C$ presents a generic constant independent of the meshsize or the functions appearing in the inequality.

\section{Weak Formulation}
Introducing a Lagrange multiplier \ $\lambda$ \ and applying \ Euler-Lagrange \ method to the constrained  minimization problem \eqref{mini}-\eqref{constr}, we obtain the weak formulation that seeks 
$u\in H^2_{0,\Gamma_N}(\Omega)$ and $\lambda \in L^2(\Omega)$ satisfying
\begin{equation}\label{weakform}
    \begin{split}
        \langle u-c_0, v\rangle_{\Gamma_O}+\alpha \langle \partial_n u, \partial_n v\rangle_{\Gamma_C}+(-\Delta v+v, \lambda)= 0,  \forall  v\in  H^2_{0,\Gamma_N}(\Omega), 
     \end{split}
\end{equation}
\begin{equation}\label{weakform2}
    \begin{split}      
    (-\Delta u+u, w)= (f, w),\qquad \forall w \in L^2(\Omega),
    \end{split}
\end{equation}
where $ 
H^2_{0,\Gamma_N}(\Omega)=\{v\in H^2(\Omega): \partial_n v|_{\Gamma_N}=0\}$.

\begin{lemma}\label{unilem} \cite{gt} Assume that $\Omega$ is an open bounded and connected domain in $\mathbb R^2$ with Lipschitz continuous boundary $\partial\Omega$. Assume that   $\Gamma_O$ is a non-trivial portion of $\partial\Omega$.  Then,  the solutions of the following  problem, if they exist, are unique  
\begin{equation*}
\begin{split}
-\Delta u+u=&0, \qquad\mbox{\ in}\quad \Omega,\\
  u=&0,\qquad \mbox{on}\quad   \Gamma_O,\\
 \partial_\bn u =&0, \qquad\mbox{on}\quad  \partial\Omega. 
\end{split}
\end{equation*}
\end{lemma} 

\begin{lemma}
   Assume $\Gamma_O$ is a nontrivial portion of the boundary $\partial\Omega$.   The solution of the   weak formulation \eqref{weakform}-\eqref{weakform2} is unique.
\end{lemma}
\begin{proof}
Since the number of equations is the same as the number of unknowns in the  weak formulation \eqref{weakform}-\eqref{weakform2}, the solution existence is equivalent to the uniqueness. To verify the uniqueness, we consider the  homogenerous data with $f=0$ and $c_0=0$. The weak formulation \eqref{weakform}-\eqref{weakform2} could be rewritten as: Find $u\in H^2_{0,\Gamma_N}(\Omega)$   and $\lambda \in L^2(\Omega)$ satisfying
\begin{equation}\label{weak3}
    \begin{split}
          \langle u, v\rangle_{\Gamma_O}+\alpha \langle \partial_n u, \partial_n v\rangle_{\Gamma_C}+(-\Delta v+v, \lambda)=&0, \forall  v\in H^2_{0,\Gamma_N}(\Omega),\\
        (-\Delta u+u, w)=&0, \forall w \in L^2(\Omega).
    \end{split}
\end{equation}
Letting $v=u$ and $w=\lambda$ gives
$$
  \langle u, u\rangle_{\Gamma_O}+\alpha \langle \partial_n u, \partial_n u\rangle_{\Gamma_C}=0.
$$
Therefore, we have 
\begin{equation*}
    \begin{split}
        -\Delta u+u=&0, \qquad \text{in}\ \Omega,\\
        u=&0, \qquad \text{on}\ \Gamma_O,\\
        \partial_n u=&0, \qquad\text{on}\ \Gamma_C\cap\Gamma_N=\partial\Omega.
    \end{split}
    \end{equation*} 
  Using Lemma \ref{unilem}, this gives $u=0$ in $\Omega$.  Therefore, the first equation in \eqref{weak3} gives
    $$
    (-\Delta v+v, \lambda)=0, \qquad \forall v\in H^2_{0,\Gamma_N}(\Omega).
    $$ This yields
    $
  \lambda=0 $ in $\Omega$ by taking  some $v$ such that
$-\Delta v+v=\lambda$.

This completes the proof of the lemma. 
\end{proof}


\section{Weak Laplacian and Discrete Weak Laplacian}\label{Section:Hessian}

In this section, we shall briefly review the definition of weak Laplacian and its discrete version  proposed in \cite{wang}.

Let $T$ be a polygonal  element with boundary $\partial
T$. A weak function on $T$ refers to a triplet $v=\{v_0,
v_b, v_n \bn\}$ such that $v_0\in L^2(T)$, $v_b\in L^{2}(\partial
T)$ and $v_n\in L^{2}(\partial T)$. Here $\bn$ is the unit outward normal
direction on $\partial T$. The first component $v_0$ represents the
``value" of $v$ in the interior of $T$, and the rest, namely $v_b$ and $v_n$, are reserved for the boundary information of $v$. In application to the Laplacian operator, $v_b$ denotes the boundary value of $v$ and $v_n$ is the outward normal derivative of $v$ on $\partial T$; i.e., $v_n \approx \nabla v\cdot \bn$. In general, $v_b$ and $v_n$ are assumed to be independent of the trace of $v_0$ and
$\nabla v_0 \cdot \bn $, respectively, on $\partial T$, but the special cases of $v_b= v_0|_{\partial T}$ and $v_n= (\nabla v_0 \cdot \bn)|_{\partial T}$ are completely legitimate, and when this happens, the function $v=\{v_0, v_b, v_n \bn\}$ is uniquely determined by $v_0$ and shall be simply denoted as $v=v_0$.

Denote by $W(T)$ the space of all weak functions on $T$; i.e.,
\begin{equation}\label{2.1}
W(T)=\{v=\{v_0,v_b, v_n\bn\}: v_0\in L^2(T), v_b\in L^{2}(\partial
T), v_n\in L^{2}(\partial T) \}.
\end{equation}

The weak Laplacian, denoted by $\Delta_{w}$, is a linear
operator from $W(T)$ to the dual of $H^{2}(T)$ such that for any
$v\in W(T)$, $\Delta_w v$ is a bounded linear functional on $H^2(T)$
defined by
\begin{equation}\label{2.3}
 \langle\Delta _{ w}v,\varphi\rangle_T=(v_0,\Delta \varphi)_T-
 \langle v_b,\nabla \varphi\cdot \bn \rangle_{\partial T}+
 \langle v_n,\varphi  \rangle_{\partial T},\quad \forall \varphi\in H^2(T),
 \end{equation}
where the left-hand side of (\ref{2.3})
represents the action of the linear functional $\Delta_w v$ on
$\varphi\in H^2(T)$.

For any non-negative integer $r\ge 0$, let $P_r(T)$ be the space of
polynomials on $T$ with total degree $r$ and less. A discrete weak
Laplacian on $T$, denoted by $\Delta_{w,r,T}$, is a linear operator
from $W(T)$ to $P_r(T)$ such that for any $v\in W(T)$,
$\Delta_{w,r,T}v$ is the unique polynomial in $P_r(T)$ satisfying
\begin{equation}\label{2.4}
 (\Delta _{w,r,T} v,\varphi)_T=(v_0,\Delta \varphi)_T-
 \langle v_b,\nabla \varphi\cdot \bn \rangle_{\partial T}+
 \langle v_n,\varphi  \rangle_{\partial T},\quad \forall \varphi \in P_r(T).
 \end{equation}
For a smooth $v_0\in
H^2(T)$,  applying the usual integration by parts to the first
term on the right-hand side of (\ref{2.4})  gives
\begin{equation}\label{2.4new}
 (\Delta _{w,r,T} v,\varphi)_T=(\Delta v_0,  \varphi)_T+
 \langle v_0- v_b,\nabla \varphi\cdot \bn \rangle_{\partial T}-
 \langle  \nabla v_0\cdot \bn - v_n,\varphi  \rangle_{\partial T}.
 \end{equation}

\section{Weak Galerkin Algorithm}\label{Section:WGFEM}
Let ${\cal T}_h$ be a finite element partition of the domain
$\Omega\subset \mathbb R^2$ into polygons. Assume that ${\cal
T}_h$ is shape regular in the sense described as in \cite{wy3655}.
Denote by ${\mathcal E}_h$ the set of all edges  in
${\cal T}_h$ and ${\mathcal E}_h^0={\mathcal E}_h \setminus
\partial\Omega$ the set of all interior edges. Denote
by $h_T$ the diameter of $T\in {\cal T}_h$ and $h=\max_{T\in {\cal
T}_h}h_T$ the meshsize of the finite element partition ${\cal
T}_h$.

Let $k\geq 2$.  For any element $T\in\T_h$, define a local
weak finite element space as follows:
$$
V(k,T)=\{\{v_0,v_b,v_n\bn\}:\ v_0\in P_k(T),v_b\in P_k(e), v_n\in
P_{k-1}(e), e\subset \partial T\}.
$$
By patching $V(k,T)$ over all the elements $T\in {\cal T}_h$ through
a common value $v_b$ and $v_n\bn$ on the interior interface $\E_h^0$,
we obtain a global weak finite element space; i.e.,
$$
V_h=\big\{\{v_0,v_b,v_n\bn\}:\ \{v_0,v_b,v_n\bn\}|_T\in V(k,T),
\forall T\in {\cal T}_h \big\}.
$$

For any interior edge/face $e\in\E_h^0$,  there
exist two elements $T_1$ and $T_2$ sharing $e$ as a common
edge/face. Thus, any finite element function $v\in V_h$ would
satisfy the following property
$$
(v_n \bn)|_{\partial T_1 \cap e} = (v_n \bn)|_{\partial T_2 \cap e},
$$
where the left-hand side (respectively, right-hand side) stands for the value of $v_n\bn$ as seen from the element $T_1$ (respectively, $T_2$). As the two normal directions are opposite to
each other, it follows that
$$
(v_n)|_{\partial T_1 \cap e} + (v_n)|_{\partial T_2 \cap e}= 0.
$$
 
 We introduce an auxiliary finite element space as follows:
$$
W_h=\{\sigma: \ \sigma|_T\in P_{r}(T), r=k-2 \ \text{or}\ r=k-1,  T\in {\cal T}_h\}.
$$
Denote by ${\cal Q}_h$ the $L^2$ projection operator onto the finite element
space $W_h$. For any $v\in V_h$, the discrete weak Laplacian, denoted by $\Delta_{w,h}v$, is computed by applying the discrete weak Laplacian $\Delta_{w,r,T}$ to $v$ locally on each element; i.e.,
$$
(\Delta_{w,h}v)|_T = \Delta _{w,r,T} (v|_T).
$$  


Let $V_h^0$ be subspace of $V_h$  such that
$$
V_h^0=\big\{\{v_0,v_b,v_n\bn\}\in V_h:  v_n|_{\Gamma_N}=0\big\}.
$$

\medskip

We introduce two bilinear forms  as follows
\begin{equation*}
\begin{split}
 s(u, v)=&\sum_{T\in {\cal T}_h}s_T(u,v),\qquad u, v\in V_h,\\
 b(v, \lambda)=&\sum_{T\in {\cal T}_h}b_T(v, \lambda),\qquad   v\in V_h, \lambda\in W_h,
\end{split}
\end{equation*}
where
\an{\label{s-h} & \ad{ 
 s_T(u, v)&= h_T^{-3}\int_{\partial T}
(u_0-u_b)(v_0-v_b)ds \\ 
  &\quad \ + h_T^{-1}\int_{\partial T} (\nabla
u_0 \cdot \bn-u_n ) (\nabla v_0\cdot\bn -v_n)ds\\ 
&\quad \ +(u_0-{\cal Q}_h u_0, v_0-{\cal Q}_h v_0)_T,}\\
 & b_T(v, \lambda)  = (-\Delta_{w} v+v_0, \lambda_h)_T. \label{b-T} }

The weak Galerkin scheme for the weak formulation \eqref{weakform}-\eqref{weakform2} of the optimal control model problem \eqref{model}-\eqref{model2} is as follows:
\begin{algorithm}
    Find $u_h=\{u_0, u_b, u_n\bn\}\in V_h^0$ and
$\lambda_h\in W_h$ such that
\begin{equation}\label{wg}
    \begin{split}
   s(u_h, v)+ \langle u_b, v_b\rangle_{\Gamma_O}+\alpha  \langle u_n, v_n\rangle_{\Gamma_C}+b(v,\lambda_h)
    = \langle c_0, v_b\rangle_{\Gamma_O}, \forall  v \in V^0_h,\\
  \end{split}
\end{equation} 
\begin{equation}\label{wg2}
    \begin{split}
b(u_h, \sigma)  = (f,\sigma), \quad \forall \sigma\in W_h.      
    \end{split}
\end{equation} 

\end{algorithm}

\begin{theorem}\label{thmunique1}    The weak Galerkin finite element algorithm
(\ref{wg})-\eqref{wg2}  has one and only one solution pair
$(u_h; \lambda_h) \in V_h^0 \times W_h$.
%
\end{theorem}

\begin{proof} Since the number of equations is the same as the
number of unknowns in the system of linear equations
(\ref{wg})-\eqref{wg2}, the solution existence is
equivalent to the solution uniqueness. To verify the uniqueness, we consider the optimal control model problem \eqref{model}-\eqref{model2} with homogeneous data (i.e., $f\equiv 0$, and $c_0\equiv 0$). We choose  $v=u_h$  and $\sigma=\lambda_h$ in (\ref{wg})-\eqref{wg2}   to obtain
\a{
s(u_h, u_h)+ \langle u_b, u_b\rangle_{\Gamma_O}&+\alpha  \langle u_n, u_n\rangle_{\Gamma_C}
         =0,\\
       (-\Delta_{w,h} u_h+u_0, \sigma)&=0 \qquad\forall \sigma\in W_h, }
which, from \eqref{s-h} and \eqref{b-T}, leads to   
\an{\label{0} \ad{ u_0|_e&=u_b|_e, \qquad &&\forall e\in \E_h, \\
                 \nabla u_0\cdot \bn|_e&=u_n|_e, \qquad &&\forall e\in \E_h, \\
                    u_0|_T &= {\cal Q}_h u_0|_T, \qquad &&\forall T\in \T_h, \\
                   u_b|_e&=0, \qquad &&\forall e\subset \Gamma_O, \\
                   u_n|_e&=0, \qquad &&\forall e\subset \Gamma_C, \\
      (- \Delta_{w,h} u_0 + u_0)|_T &=0, \qquad &&\forall T\in \T_h. } }
We note that the last equality holds because both $- \Delta_{w,h} u_0$ and $u_0$ are shown 
   in the space $W_h$.
By using (\ref{2.4new}),  we obtain
\begin{equation}\label{EQ:Nov26:001}
(-\Delta u_0+u_0, \sigma)_T = (-\Delta_{w,h} u_h+u_0, \sigma)_T=0,\qquad \forall \sigma\in
W_h.
\end{equation} As $\Delta
u_0|_T\in P_{k-2}(T)$, $\Delta_{w,h}
u_0|_T\in P_{k-2}(T)$ and $ 
u_0|_T\in P_{r}(T)$ on each element $T$, we
then have from (\ref{EQ:Nov26:001}) and \eqref{0}, 
$$
-\Delta u_0+u_0 = 0, \qquad \mbox{in} \ \Omega,
$$
which, from Lemma \ref{unilem}, and together with the fact that  $u_0=0$ on $\Gamma_O$ and $\nabla u_0\cdot \bn=0$ on $\Gamma_C\cup\Gamma_N=\partial\Omega$, yields $u_0 \equiv 0$ in $\Omega$. Using \eqref{0} further gives $u_h \equiv 0$ in $\Omega$.


It remains to show that $\lambda_h\equiv 0$ in $\Omega$. To this end, from   (\ref{wg}) and the fact that $u_h\equiv 0$ in $\Omega$ we have
$$
\sum_{T\in\T_h}(-\Delta_{w,h} v+v_0, \lambda_h)_T = 0,\qquad \forall v\in V_h^0.
$$
It follows from (\ref{2.4}) that
\begin{eqnarray*}
  0  
&=&\sum_{T\in\T_h}(-\Delta_{w,h} v+v_0, \lambda_h)_T\\ 
&=&\sum_{T\in\T_h} (v_0, -\Delta \lambda_h+\lambda_h)_T +\langle
v_b,\nabla \lambda_h\cdot \bn \rangle_{\partial T} -
 \langle v_n,\lambda_h\rangle_{\partial T} \\
 &=&\sum_{T\in\T_h}(v_0, -\Delta \lambda_h+\lambda_h)_T + \sum_{e\in\E_h }
 \langle v_b, \ljump{\nabla\lambda_h}\rjump\cdot\bn_e\rangle_e \\ & & \ -
 \sum_{e\in\E_h/ \Gamma_N } \langle v_n, \ljump{\lambda_h}\rjump\rangle_e,
\end{eqnarray*}
for all $v\in V_h^0$, where we have used the fact that $v_n=0$ on $\Gamma_N$. Here, $\ljump{ \lambda_h}\rjump$ is the jump across the edge   $e\in {\cal E}_h$; more precisely, it is defined as $\ljump{\lambda_h}\rjump=\lambda_h|_{T_1}-\lambda_h|_{T_2}$ whereas $e$ is the shared edge of the elements $T_1$ and $T_2$ and $\ljump{\lambda_h}\rjump=\lambda_h$ whereas $e\subset\partial\Omega$. The order of $T_1$ and $T_2$ is non-essential as long as the difference
is taken in a consistent way in all the formulas. By letting $v_0=- \Delta \lambda_h+\lambda_h$ on each
element $T$ and $v_b=  \ljump{\nabla\lambda_h}\rjump\cdot\bn_e$ on each edge  $e\in \E_h$ and
$v_n=-\ljump{\lambda_h}\rjump$ on each $e\in \E_h/\Gamma_N$ in the above equation, we obtain
\begin{eqnarray}\label{EQ:Nov26:002}
-\Delta \lambda_h+\lambda_h &=& 0,\qquad \mbox{on each } T\in \T_h,\\
\ljump{\nabla\lambda_h}\rjump\cdot\bn_e & = & 0,\qquad \mbox{on each edge } \ e\in \E_h,
\label{EQ:Nov26:003}\\
\ljump{\lambda_h}\rjump & = & 0,\qquad \mbox{on each edge } \
e\in \E_h/\Gamma_N.\label{EQ:Nov26:004}
\end{eqnarray}
The equations (\ref{EQ:Nov26:003}) and (\ref{EQ:Nov26:004}) indicate
that $\lambda_h\in C^1(\Omega)$ and
$\nabla\lambda_h\cdot\bn=0$ on $\partial\Omega$ and $\lambda_h=0$ on  $\Gamma_C$, where   $\Gamma_C= \partial\Omega/\Gamma_N$.  
Thus, the equation (\ref{EQ:Nov26:002}) holds true in the whole
domain $\Omega$. Thus, from Lemma \ref{unilem}, we have $\lambda_h\equiv 0$ in $\Omega$. 
This completes the proof of the theorem.
\end{proof}

\section{Stability Analysis}\label{Section:stability}
For any $v\in V_h$, we introduce a semi-norm as follows
\begin{equation}\label{norm}
\3bar v \3bar =  
\Big(s(v, v)+\langle v_b, v_b\rangle_{\Gamma_O}+\alpha \langle v_n, v_n\rangle_{\Gamma_C}\Big)^{\frac{1}{2}}.
\end{equation}
For any $\lambda \in W_h$, we introduce the
following $L^2$ norm
\begin{equation}\label{EQ:NormInWh}
\|\lambda\|=  \Big(\sum_{T\in\T_h} \|\lambda\|_T^2\Big)^{1/2}.
\end{equation}

On each element $T\in\T_h$, denote by $Q_0$, $Q_b$ and $Q_n$ the $L^2$ projection
operators onto $P_k(T)$, $P_{k}(e)$ and $P_{k-1}(e)$ respectively. For any $\theta\in H^2(\Omega)$, denote by $Q_h
\theta$ the $L^2$ projection onto the weak finite element space
$V_h$ such that on each element $T$,
$$
Q_h\theta=\{Q_0\theta,Q_b\theta, Q_n(\nabla \theta\cdot \bn)\bn\}.
$$
The following commutative property
holds true \cite{wang}:
\begin{equation}\label{EQ:CommutativeProperty}
\Delta _{w,h}(Q_h \theta) = {\cal Q}_h(\Delta \theta),\qquad \theta
\in H^2(T).
\end{equation}

Assume that the finite element partition ${\cal T}_h$ is
shape-regular. Thus, on each $T\in {\cal T}_h$, the
following trace inequality holds true \cite{wy3655}:
\begin{equation}\label{tracein}
 \|\phi\|^2_{\partial T} \lesssim   h_T^{-1}\|\phi\|_T^2+h_T \|\nabla
 \phi\|_T^2,\quad \phi\in H^1(T).
\end{equation}
If $\phi$ is additionally a polynomial function on the element $T\in
{\cal T}_h$, we have from (\ref{tracein}) and the inverse inequality
(see \cite{wy3655} for details on arbitrary polygonal elements) that
\begin{equation}\label{tracenew}
 \|\phi\|^2_{\partial T} \lesssim  h_T^{-1}\|\phi\|_T^2.
\end{equation}
  
\begin{lemma} Assume that the finite element partition ${\cal T}_h$ is
shape-regular.  Then, for any $0\leq s \leq 2$ and
$1\leq m \leq k$, one has
\begin{equation}\label{error1}
 \sum_{T\in {\cal T}_h}h_T^{2s}\|u-Q_0u\|^2_{s,T}\lesssim   h^{2(m+1)}\|u\|^2_{m+1},
\end{equation}
\begin{equation}\label{error3}
 \sum_{T\in {\cal T}_h}h_T^{2s}\|u-{\cal Q}_hu\|^2_{s,T}\lesssim   h^{2m}\|u\|^2_{m}.
\end{equation}
 \end{lemma}

\begin{lemma} \label{leminf} [inf-sup condition] For any $\lambda \in W_h$, there exists  a weak function $v\in V_h^0$ satisfying
\begin{align}\label{inf1}
 (-\Delta_{w,h} v+v_0, \lambda)&  = \|\lambda\|^2,  \\
 \3bar v\3bar  & \lesssim \|\lambda\|.\label{inf2}
\end{align}
\end{lemma}

\begin{proof}  
Let $\Phi$ be the solution of the following auxiliary problem
\begin{equation}\label{auxi}
    \begin{split}
        -\Delta \Phi+\Phi=&\lambda, \qquad\text{in}\ \Omega,\\
        \Phi=&0,\qquad\text{on}\ \partial\Omega.
    \end{split}
\end{equation}
We assume that the auxiliary problem \eqref{auxi} has the following $H^2$- regularity estimate 
\begin{equation}\label{regu1}
    \|\Phi\|_2\lesssim\|\lambda\|.
\end{equation}
As to \eqref{inf1}, letting $v=Q_h \Phi=\{Q_0\Phi,Q_b\Phi, Q_n(\nabla \Phi\cdot \bn)\bn\}$ and using the commutative property \eqref{EQ:CommutativeProperty}, we have
\begin{equation*}
    \begin{split}
        (-\Delta_w v+v_0, \lambda)=& \sum_{T\in {\cal T}_h}  (-\Delta_w Q_h \Phi+Q_0\Phi, \lambda)_T\\
        =&\sum_{T\in {\cal T}_h}  (-  {\cal Q}_h \Delta \Phi+Q_0\Phi, \lambda)_T\\
         =&\sum_{T\in {\cal T}_h}  (-   \Delta \Phi+ \Phi, \lambda)_T\\
         =&\|\lambda\|^2.
    \end{split}
\end{equation*}

Note that when the trace inequality \eqref{tracenew} is applied to the whole domain $\Omega$, the mesh-size is $h={\cal O}(1)$. As to \eqref{inf2}, using the trace inequalities \eqref{tracein}-\eqref{tracenew}, the estimates \eqref{error1}-\eqref{error3}, and the regularity assumption \eqref{regu1}, we have 
\begin{equation*}
    \begin{split}
&  \3bar v\3bar^2 
 =  \3bar Q_h \Phi\3bar^2\\
 =&\sum_{T\in {\cal T}_h}h_T^{-3} \|
 Q_0 \Phi-Q_b \Phi\|^2_{\partial T}+ h_T^{-1} \|\nabla
Q_0 \Phi \cdot \bn-Q_n (\nabla \Phi \cdot\bn)\|^2_{\partial T} \\
&+\|Q_0\Phi-{\cal Q}_h Q_0\Phi\|^2_T+\|Q_b \Phi \|^2_{\Gamma_O}+\alpha\|Q_n(\nabla \Phi\cdot\bn)\|_{\Gamma_C}^2\\
\lesssim &  \sum_{T\in {\cal T}_h} h_T^{-3} \|
 Q_0 \Phi-  \Phi\|^2_{\partial T}+ h_T^{-1} \|\nabla
Q_0 \Phi \cdot \bn-  (\nabla \Phi \cdot\bn)\|^2_{\partial T} \\
&+\|Q_0\Phi-{\cal Q}_h Q_0\Phi\|^2_{T}+ \|Q_b \Phi \|^2+ \|Q_n(\nabla \Phi\cdot\bn)\|^2\\
\lesssim &  \sum_{T\in {\cal T}_h} h_T^{-4} \|
 Q_0 \Phi-  \Phi\|^2_{  T}+h_T^{-2} \|
 Q_0 \Phi-  \Phi\|^2_{1,  T}+  h_T^{-2} \|
 Q_0 \Phi-  \Phi\|^2_{1,  T}\\&+   \| 
Q_0 \Phi  -  \Phi  \|^2_{2, T} +\|Q_0\Phi-{\cal Q}_h Q_0\Phi\|^2_T+ \| \Phi \|^2+\|\nabla \Phi\cdot\bn\|^2\\
\lesssim & (1+h^2)\|\Phi\|_2^2\lesssim \|\lambda\|^2.
    \end{split}
\end{equation*}

This completes the proof of the Lemma.
\end{proof}

\section{Error equations}

 Now let $(u_h;\lambda_h)\in V^0_h \times W_h$
be the numerical solution arising from the weak Galerkin
algorithm (\ref{wg})-\eqref{wg2}. Denote the error
functions by
\begin{align}\label{error}
e_h&=u_h-Q_h u,\\
\epsilon_h&=\lambda_h - {\cal Q}_h \lambda.\label{error-2}
\end{align}

Applying the usual integration by parts to \eqref{weakform}, the Lagrange multiplier $\lambda$ satisfies
\begin{equation}\label{eqnlam}
    \begin{split}
        -\Delta\lambda+\lambda=&0, \ \text{in}\ \Omega,\\
        \nabla\lambda\cdot\bn=&c_0-u,\ \text{on}\ \Gamma_O,\\
        \nabla\lambda\cdot\bn=&0,\ \text{on}\ \partial\Omega \setminus\Gamma_O,\\
        \lambda=&\alpha \partial_n u, \ \text{on}\ \Gamma_C.
    \end{split}
\end{equation}

\begin{lemma}\label{errorequa}
Let $u$ be the solution of the optimal control model problem \eqref{model}-\eqref{model2}  and $(u_h;\lambda_h)\in V^0_h \times W_h$ be
its numerical approximation arising from the weak
Galerkin algorithm (\ref{wg})-\eqref{wg2}. Then, the error functions $e_h$ and $\epsilon_h$ defined in (\ref{error})-(\ref{error-2}) satisfy the following equations
\begin{equation}\label{sehv1}
    \begin{split}
        &s(e_h, v)+\langle e_b, v_b\rangle_{\Gamma_O}+\alpha \langle e_n, v_n\rangle_{\Gamma_C}+(-\Delta_w v+v_0, \epsilon_h)_T
        \\=& \sum_{T\in {\cal T}_h} \langle \nabla v_0\cdot\bn-v_n, \lambda-{\cal Q}_h \lambda\rangle_{\partial T}-\langle v_0-v_b,  (\nabla \lambda-\nabla {\cal Q}_h \lambda) \cdot\bn\rangle_{\partial T} \\
    &   -  (v_0,  {\cal Q}_h \lambda-\lambda)_T+s(Q_hu, v), \qquad \forall v\in V_h^0,
    \end{split}
\end{equation} 
\begin{equation}\label{sehv2}
    (-\Delta_{w,h} e_h+e_0, w) = 0,\qquad\qquad\qquad \forall w\in W_h. 
\end{equation} 
\end{lemma}

\begin{proof}
From \eqref{2.4new}, the usual integration by parts, and \eqref{eqnlam}, we obtain
\begin{equation*}
    \begin{split}
 &     \sum_{T\in {\cal T}_h}  (-\Delta_w v+v_0, {\cal Q}_h \lambda)_T
    \\=&  \sum_{T\in {\cal T}_h} (-\Delta v_0, {\cal Q}_h \lambda)_T+\langle v_b-v_0, \nabla {\cal Q}_h \lambda \cdot\bn\rangle_{\partial T}\\&-\langle v_n-\nabla v_0\cdot\bn, {\cal Q}_h \lambda\rangle_{\partial T}+ (v_0,  {\cal Q}_h \lambda)_T \\
    =&  \sum_{T\in {\cal T}_h}  (v_0,-\Delta  \lambda+\lambda)_T -\langle \nabla v_0\cdot\bn-v_n, \lambda\rangle_{\partial T}+\langle v_0-v_b,  \nabla \lambda \cdot\bn\rangle_{\partial T} \\
    &+\langle v_b-v_0, \nabla {\cal Q}_h \lambda \cdot\bn\rangle_{\partial T}-\langle v_n-\nabla v_0\cdot\bn, {\cal Q}_h \lambda\rangle_{\partial T} -\langle  v_n, \lambda\rangle_{\partial T}\\
    &+\langle  v_b,  \nabla \lambda \cdot\bn\rangle_{\partial T}  + (v_0,  {\cal Q}_h \lambda-\lambda)_T \\
    =& \sum_{T\in {\cal T}_h}    -\langle \nabla v_0\cdot\bn-v_n, \lambda-{\cal Q}_h \lambda\rangle_{\partial T}+\langle v_0-v_b,  (\nabla \lambda-\nabla {\cal Q}_h \lambda) \cdot\bn\rangle_{\partial T} \\
    & + (v_0,  {\cal Q}_h \lambda-\lambda)_T-\langle  v_n, \alpha \partial_n u\rangle_{\Gamma_C}+\langle  v_b,  c_0-u\rangle_{\Gamma_O}.  
    \end{split}
\end{equation*}
This gives
\begin{equation}\label{err1}
    \begin{split}
        &\langle Q_bu, v_b\rangle_{\Gamma_O}+\alpha \langle Q_n(\nabla u\cdot\bn), v_n\rangle_{\Gamma_C}+(-\Delta_w v+v_0, {\cal Q}_h \lambda)_T
        \\=& \sum_{T\in {\cal T}_h}    -\langle \nabla v_0\cdot\bn-v_n, \lambda-{\cal Q}_h \lambda\rangle_{\partial T}+\langle v_0-v_b,  (\nabla \lambda-\nabla {\cal Q}_h \lambda) \cdot\bn\rangle_{\partial T} \\
    &  + (v_0,  {\cal Q}_h \lambda-\lambda)_T+\langle  v_b,  c_0 \rangle_{\Gamma_O}.
    \end{split}
\end{equation}
Combining   (\ref{wg}) and \eqref{err1} gives \eqref{sehv1}.

As to  (\ref{sehv2}), we use  (\ref{wg2}),   \eqref{model}, and the
commutative property (\ref{EQ:CommutativeProperty}) to obtain
\begin{eqnarray*}
 (-\Delta_{w,h} e_h+e_0, w) & = & (-\Delta_{w,h}(u_h-Q_hu)+(u_0-Q_0u), w)\\
 &= & (-\Delta_{w,h} u_h+u_0, w) +(\Delta_{w,h} Q_h u-Q_0u, w) \\
 & = & (f, w) + ({\cal Q}_h\Delta u-Q_0u, w) \\
 & = & (f,w) + ( \Delta u-u, w)\\
  & = & (f,w) - (f, w)\\
& = & 0,
\end{eqnarray*}
for all $w\in W_h$. This completes the proof of the lemma.
\end{proof} 

\section{Error estimates}

\begin{theorem} \label{theoestimate}
Let $k \geq 2$. Let $u$ be the exact solution of the optimal control model problem \eqref{model}-\eqref{model2}, and $(u_h;\lambda_h)\in V_h^0\times W_h$
be its numerical approximation arising from the weak
Galerkin algorithm (\ref{wg})-\eqref{wg2}.   Assume that the exact solution is sufficiently regular such  that $u\in
H^{k+1}(\Omega)$ and the Lagrange multiplier satisfies $\lambda\in H^{k-1}(\Omega)$. The following error estimate holds true:
 \begin{equation}\label{erres}
\3bar e_h\3bar+\|\epsilon_h\| \lesssim 
h^{k-1}(\|\lambda\|_{k-1}  +   \|u\|_{k+1}).
\end{equation}
\end{theorem}

\begin{proof} 
Note that the error function
$e_h  \in V_h^0$. We  take $v=e_h$ in the
error equation (\ref{sehv1}) to obtain
\begin{equation} 
    \begin{split}
        &s(e_h, e_h)+\langle e_b, e_b\rangle_{\Gamma_O}+\alpha \langle e_n, e_n\rangle_{\Gamma_C}+\sum_{T\in {\cal T}_h}(-\Delta_w e_h+e_0, \epsilon_h)_T
        \\=& \sum_{T\in {\cal T}_h}    \langle \nabla e_0\cdot\bn-e_n, \lambda-{\cal Q}_h \lambda\rangle_{\partial T}-\langle e_0-e_b,  (\nabla \lambda-\nabla {\cal Q}_h \lambda) \cdot\bn\rangle_{\partial T} \\
    &   -  (e_0,  {\cal Q}_h \lambda-\lambda)_T+s(Q_hu, e_h). 
    \end{split}
\end{equation} 
 
Note that    (\ref{sehv2}) implies $\sum_{T\in {\cal T}_h}(-\Delta_w e_h+e_0, \epsilon_h)_T=0$. Thus, we obtain
\begin{equation} \label{eh}
    \begin{split}
         \3bar e_h\3bar^2
         =&  \sum_{T\in {\cal T}_h}   \langle \nabla e_0\cdot\bn-e_n, \lambda-{\cal Q}_h \lambda\rangle_{\partial T} \\ & 
 -\langle e_0-e_b,  (\nabla \lambda-\nabla {\cal Q}_h \lambda) \cdot\bn\rangle_{\partial T} \\
    &   -  (e_0,  {\cal Q}_h \lambda-\lambda)_T+s(Q_hu, e_h)\\
    =&I_1+I_2+I_3+I_4.
    \end{split}
\end{equation} 

We shall estimate the four terms in the last line of \eqref{eh}. 
As to $I_1$, from the Cauchy-Schwarz inequality, the trace
inequality (\ref{tracein}) and the estimate (\ref{error3}) we have
\begin{equation}\label{i1}
    \begin{split}
   I_1\leq &\sum_{T\in {\cal T}_h}    |\langle \nabla e_0\cdot\bn-e_n, \lambda-{\cal Q}_h \lambda\rangle_{\partial T}|\\
   \lesssim & (\sum_{T\in {\cal T}_h}h_T^{-1} \| \nabla e_0\cdot\bn-e_n\|^2_{\partial T})^{\frac{1}{2}}(\sum_{T\in {\cal T}_h}h_T \| \lambda-{\cal Q}_h \lambda\|^2_{\partial T})^{\frac{1}{2}}\\
\lesssim & \3bar e_h\3bar (\sum_{T\in {\cal T}_h}  \| \lambda-{\cal Q}_h \lambda\|^2_{T}+h_T^2 \| \lambda-{\cal Q}_h \lambda\|^2_{1, T})^{\frac{1}{2}}\\
\lesssim & h^{k-1}\|\lambda\|_{k-1}\3bar e_h\3bar.
    \end{split}
\end{equation}

As to $I_2$, from the Cauchy-Schwarz inequality, the trace
inequality (\ref{tracein}) and the estimate (\ref{error3}) we have
\begin{equation}\label{i2}
    \begin{split}
  I_2\leq &\sum_{T\in {\cal T}_h}    |\langle e_0-e_b,  (\nabla \lambda-\nabla {\cal Q}_h \lambda) \cdot\bn\rangle_{\partial T}|\\
  \lesssim & (\sum_{T\in {\cal T}_h}h_T^{-3} \| e_0-e_b\|^2_{\partial T})^{\frac{1}{2}}(\sum_{T\in {\cal T}_h}h_T^3 \| (\nabla \lambda-\nabla {\cal Q}_h \lambda)\|^2_{\partial T})^{\frac{1}{2}}\\
\lesssim & \3bar e_h\3bar(\sum_{T\in {\cal T}_h}  h_T^2 \| \nabla \lambda-\nabla {\cal Q}_h \lambda\|^2_T+h_T^4  \|\nabla \lambda-\nabla {\cal Q}_h \lambda\|^2_{1, T})^{\frac{1}{2}}\\
 \lesssim & h^{k-1}\|\lambda\|_{k-1}\3bar e_h\3bar.
    \end{split}
\end{equation}

As to $I_3$, from the Cauchy-Schwarz inequality, and the estimate (\ref{error3}), we have
\begin{equation}\label{i4}
    \begin{split}
  I_3\leq &\sum_{T\in {\cal T}_h}    | (e_0, {\cal Q}_h\lambda-\lambda)|\\
  = &\sum_{T\in {\cal T}_h}    | (e_0-{\cal Q}_he_0, {\cal Q}_h\lambda-\lambda)|\\
    \lesssim & h^{k-1}\|\lambda\|_{k-1}\3bar e_h\3bar.
    \end{split}
\end{equation}

Recall that \begin{equation}\label{sterm}
\begin{split}
I_4=&s(Q_hu,e_h)\\
= &\sum_{T\in {\cal T}_h}h_T^{-3} \langle
 Q_0u-Q_bu, e_0-e_b\rangle _{\partial T}  \\
 &+ h_T^{-1} \langle  \nabla Q_0u \cdot \bn-Q_n(\nabla u \cdot \bn), \nabla  e_0\cdot\bn -e_n\rangle _{\partial T}\\&+(Q_0u-{\cal Q}_h Q_0u,  e_0-{\cal Q}_h  e_0)_T\\
 =&J_1+J_2+J_3.
\end{split}
\end{equation}
From the Cauchy-Schwarz inequality, the trace
inequality (\ref{tracein}) and the estimate (\ref{error1}) we have
\begin{equation}\label{sterm1}
\begin{split}
  J_1\leq &\Big| \sum_{T\in {\cal T}_h}h_T^{-3} \langle
 Q_0u-Q_bu, e_0-e_b\rangle _{\partial T} \Big|\\
\lesssim & \Big(  \sum_{T\in {\cal T}_h}h_T^{-3}\|
 Q_0u- Q_bu\|^2_{\partial T}\Big)^{\frac{1}{2}}  \Big(  \sum_{T\in {\cal T}_h}h_T^{-3}\|  e_0-e_b\|^2_{\partial T}\Big)^{\frac{1}{2}}  \\
\lesssim & \Big(  \sum_{T\in {\cal T}_h}h_T^{-4}\|
 Q_0u- u\|^2_{T}+h_T^{-2}\|
 Q_0u- u\|^2_{1, T}\Big)^{\frac{1}{2}} \3bar e_h\3bar \\
 \lesssim & h^{k-1} \|u\|_{k+1} \3bar  e_h\3bar.
 \end{split}
\end{equation}
Analogously,  $J_2$ and $J_3$
can be bounded as follows 
\begin{equation} 
\begin{split}
J_2\leq &\Big| \sum_{T\in {\cal T}_h}h_T^{-1} \langle\nabla Q_0u \cdot \bn-Q_n(\nabla u \cdot \bn), \nabla  e_0\cdot\bn -e_n\rangle _{\partial T}\Big|\\
\lesssim&h^{k-1} \|u\|_{k+1}  \3bar  e_h\3bar,
\end{split}
\end{equation} and
\begin{equation} \label{sterm2}
\begin{split}
J_3\leq  \Big| \sum_{T\in {\cal T}_h}(Q_0u-{\cal Q}_h Q_0u,  e_0-{\cal Q}_h  e_0)_T \Big| 
\lesssim h^{k-1} \|u\|_{k-1}  \3bar  e_h\3bar.
\end{split}
\end{equation}

Substituting (\ref{sterm1})-(\ref{sterm2}) into  (\ref{sterm})
gives
\begin{equation}\label{EQ:Estimate_4_RHT}
I_4\leq\Big| s(Q_hu,e_h)\Big|  \lesssim  h^{k-1}   \|u\|_{k+1} \3bar  e_h\3bar.
\end{equation}

Substituting the estimates \eqref{i1}, \eqref{i2}, \eqref{i4} and \eqref{EQ:Estimate_4_RHT} into  \eqref{eh} gives

\begin{equation}\label{ehess}
\3bar e_h\3bar \lesssim  h^{k-1}\|\lambda\|_{k-1} +h^{k-1}   \|u\|_{k+1}.   
\end{equation}

From the inf-sup condition (\ref{inf1})-\eqref{inf2}, 
for any $\epsilon_h \in W_h$, there exists  a weak function $v\in V_h^0$ satisfying
\begin{align}\label{inf3}
 (-\Delta_{w,h} v+v_0, \epsilon_h )&  = \|\epsilon_h \|^2,  \\
 \3bar v\3bar  & \lesssim \|\epsilon_h \|.\label{inf4}
\end{align}

On the other hand, the error equation \eqref{sehv1} implies
\begin{equation}\label{sehv3}
    \begin{split}
        &(-\Delta_w v+v_0, \epsilon_h)_T
        \\=& \sum_{T\in {\cal T}_h}    \langle \nabla v_0\cdot\bn-v_n, \lambda-{\cal Q}_h \lambda\rangle_{\partial T} \\
     & \
   -\langle v_0-v_b,  (\nabla \lambda-\nabla {\cal Q}_h \lambda) \cdot\bn\rangle_{\partial T} \\
    & \  - (v_0,  {\cal Q}_h \lambda-\lambda)_T+s(Q_hu, v)-(s(e_h, v)\\
    & \ +\langle e_b, v_b\rangle_{\Gamma_O}+\alpha \langle e_n, v_n\rangle_{\Gamma_C}).
    \end{split}
\end{equation} 
 
Similar to the estimate of $J_1$, we have
\begin{equation}\label{b1}
  |\sum_{T\in {\cal T}_h}    \langle \nabla v_0\cdot\bn-v_n, \lambda-{\cal Q}_h \lambda\rangle_{\partial T} |
  \lesssim h^{k-1} \|\lambda\|_{k-1}\3bar v\3bar.
\end{equation}

\begin{equation}\label{b2}
  |\sum_{T\in {\cal T}_h}    \langle v_0-v_b,  (\nabla \lambda-\nabla {\cal Q}_h \lambda) \cdot\bn\rangle_{\partial T} |
  \lesssim h^{k-1} \|\lambda\|_{k-1}\3bar v\3bar.
\end{equation}

\begin{equation}\label{b3}
\begin{split}
    |\sum_{T\in {\cal T}_h}   (v_0,  {\cal Q}_h \lambda-\lambda)_T | &=  |\sum_{T\in {\cal T}_h}   (v_0-{\cal Q}_h v_0,  {\cal Q}_h \lambda-\lambda)_T |\\
    & 
  \lesssim h^{k-1} \|\lambda\|_{k-1}\3bar v\3bar. 
\end{split}
\end{equation}

Similar to \eqref{EQ:Estimate_4_RHT}, we have
\begin{equation}\label{b4}
\Big| s(Q_hu,v)\Big|  \lesssim  h^{k-1}   \|u\|_{k+1}  \3bar  v\3bar.
\end{equation}

Using Cauchy-Schwartz inequality gives
\begin{equation}\label{b5}
 |s(e_h, v)+\langle e_b, v_b\rangle_{\Gamma_O}+\alpha \langle e_n, v_n\rangle_{\Gamma_C} |\lesssim\3bar e_h\3bar \3bar v\3bar .
\end{equation}

 Substituting \eqref{b1}-\eqref{b5} into \eqref{sehv3} and using \eqref{inf3}-\eqref{inf4} yields 
 \begin{equation*}
     \begin{split}
\|\epsilon_h\|^2&\lesssim 
(h^{k-1}\|\lambda\|_{k-1} +h^{k-1}   \|u\|_{k+1} +\3bar e_h\3bar) \3bar v\3bar\\
&\lesssim 
(h^{k-1}\|\lambda\|_{k-1} +h^{k-1}   \|u\|_{k+1} +\3bar e_h\3bar) \|\epsilon_h\|,  \end{split}
 \end{equation*}
  which, together with   \eqref{ehess}, leads to
  \begin{equation*}
     \begin{split}
\|\epsilon_h\| \lesssim 
h^{k-1}\|\lambda\|_{k-1} +h^{k-1}   \|u\|_{k+1}. 
\end{split}
 \end{equation*}
  
  This
completes the proof of the theorem.
\end{proof}

\section{Error Estimates in a weak $L^2$ topology}

To establish an error estimate for the WG scheme \eqref{wg}-\eqref{wg2} in a $L^2$-related topology,  we consider the dual problem of seeking $\Phi$ satisfying
\begin{equation}\label{dual}
    \begin{split}
    -\Delta \Phi+\Phi =&e_0, \qquad\text{in}\ \Omega, \\ 
        \Phi =&0, \qquad \text{on}\ \Gamma_C, \\  
        \nabla \Phi\cdot\bn =&0, \qquad \text{on} \ \partial\Omega.    
    \end{split}
\end{equation} 
        
We assume the solution of the dual problem \eqref{dual} has $H^2$ regularity in the sense that
\begin{equation}\label{regula}
    \|\Phi\|_2\lesssim\|e_0\|.
\end{equation}

\begin{theorem} Let $k\geq 2$. Recall that $r=k-1$ or $r=k-2$. We take $r=1$ for the lowest order $k=2$. Let $u$ be the exact solution of the optimal control model problem \eqref{model}-\eqref{model2}, and $(u_h;\lambda_h)\in V_h^0\times W_h$
be its numerical approximation arising from the weak
Galerkin algorithm (\ref{wg})-\eqref{wg2}. Assume that the exact solution is sufficiently regular in the sense that $u\in
H^{k+1}(\Omega)$ and the dual variable satisfies $\lambda\in  H^{k-1}(\Omega)$. Assume that the dual problem \eqref{dual} has the $H^2$ regularity property \eqref{regula}.
Then, the following error estimate holds true
    \begin{equation}
        \|e_0\|\lesssim
h^{k+1}(\|\lambda\|_{k-1} +  \|u\|_{k+1}).
    \end{equation}
\end{theorem}
\begin{proof}
    Testing \eqref{dual} by $e_0$ gives
    \begin{equation}\label{eq2}
        \begin{split}
            &\quad \ \|e_0\|^2\\
  &= \sum_{T\in {\cal T}_h}(-\Delta\Phi+\Phi, e_0)_T\\
            &=\sum_{T\in {\cal T}_h} -(\Phi, \Delta e_0)_T
 +\langle \Phi, \nabla e_0\cdot\bn \rangle_{\partial T}-\langle \nabla \Phi\cdot\bn, e_0\rangle_{\partial T}+(\Phi, e_0)_T\\
            & = \sum_{T\in {\cal T}_h} -({\cal Q}_h\Phi, \Delta e_0)_T+\langle \Phi, \nabla e_0\cdot\bn-e_n\rangle_{\partial T}\\
    & \qquad -\langle \nabla \Phi\cdot\bn, e_0-e_b\rangle_{\partial T} +(\Phi, e_0)_T,
        \end{split}
    \end{equation}
    where we used the usual integration by parts, the fact that \ 
   $\sum_{T\in {\cal T}_h}  \langle \Phi,  e_n\rangle_{\partial T}$ $
 = \langle \Phi,    e_n\rangle_{\partial \Omega}=0$ since  $e_n=0$ on $\Gamma_N$ and $\Phi=0$ on $\Gamma_C$, and the fact that $\sum_{T\in {\cal T}_h} \langle \nabla \Phi\cdot\bn, e_b\rangle_{\partial T}= \langle \nabla \Phi\cdot\bn, e_b\rangle_{\partial \Omega}=0$ since $\nabla\Phi\cdot\bn=0$ on $\partial\Omega$. 

Letting $\varphi={\cal Q}_h\Phi$ and $v=e_h$ in \eqref{2.4new} gives
     \begin{equation} \label{eq1}
     \begin{split}
         (\Delta _{w} e_h, {\cal Q}_h\Phi)_T&=(\Delta e_0,  {\cal Q}_h\Phi)_T+
 \langle e_0- e_b,\nabla {\cal Q}_h\Phi\cdot \bn \rangle_{\partial T}\\&-
 \langle  \nabla e_0\cdot \bn - e_n,{\cal Q}_h\Phi \rangle_{\partial T}.
     \end{split}
 \end{equation}

Substituting \eqref{eq1} into \eqref{eq2} gives
\begin{equation}\label{a3}
    \begin{split}
 & \|e_0\|^2 \\
 =&\sum_{T\in {\cal T}_h} - (\Delta _{w} e_h, {\cal Q}_h\Phi)_T +
 \langle e_0- e_b,\nabla {\cal Q}_h\Phi\cdot \bn \rangle_{\partial T}\\
 &-
 \langle  \nabla e_0\cdot \bn - e_n,{\cal Q}_h\Phi \rangle_{\partial T} +\langle \Phi, \nabla e_0\cdot\bn-e_n\rangle_{\partial T}\\&-\langle \nabla \Phi\cdot\bn, e_0-e_b\rangle_{\partial T}+(\Phi, e_0)_T\\
 =&\sum_{T\in {\cal T}_h}     (-\Delta _{w} e_h+e_0, {\cal Q}_h \Phi)_T  +
 \langle e_0- e_b,\nabla ({\cal Q}_h\Phi-\Phi)\cdot \bn \rangle_{\partial T}\\
 &-
 \langle  \nabla e_0\cdot \bn - e_n,{\cal Q}_h\Phi-\Phi \rangle_{\partial T}+(\Phi-{\cal Q}_h\Phi, e_0)_T\\
 =&\sum_{T\in {\cal T}_h}  
 \langle e_0- e_b,\nabla ({\cal Q}_h\Phi-\Phi)\cdot \bn \rangle_{\partial T}\\
  & \ -
 \langle  \nabla e_0\cdot \bn - e_n,{\cal Q}_h\Phi-\Phi \rangle_{\partial T}
  + (\Phi-{\cal Q}_h\Phi, e_0)_T,  
    \end{split}
\end{equation}
 where we used the error equation \eqref{sehv2}. 

 Using the Cauchy-Schwartz inequality, the trace inequality \eqref{tracein}, \eqref{error3}, and the regularity assumption \eqref{regula}, we have
 \begin{equation}\label{a1}
     \begin{split}
  &  |\sum_{T\in {\cal T}_h}  
 \langle e_0- e_b,\nabla ({\cal Q}_h\Phi-\Phi)\cdot \bn \rangle_{\partial T}|\\
 \lesssim& \Big(  \sum_{T\in {\cal T}_h}  h_T^{-3}\|e_0- e_b\|^2_{\partial T} \Big)^{\frac{1}{2}
 }\Big(  \sum_{T\in {\cal T}_h} h_T^3\|\nabla ({\cal Q}_h\Phi-\Phi)\cdot \bn \|^2_{\partial T}\Big)^{\frac{1}{2}}\\
\lesssim & \3bar e_h\3bar \Big(  \sum_{T\in {\cal T}_h} h_T^2\|\nabla ({\cal Q}_h\Phi-\Phi)\cdot \bn \|^2_{ T}+h_T^4\|\nabla ({\cal Q}_h\Phi-\Phi)\cdot \bn \|^2_{ 1, T}\Big)^{\frac{1}{2}}\\
\lesssim &   h^2 \3bar e_h\3bar  \|\Phi\|_2 \\
\lesssim &  h^2 \3bar e_h\3bar  \|e_0\|.
    \end{split}
 \end{equation} 

 Similarly, \ using the Cauchy-Schwarz inequality, \ the trace inequality \eqref{tracein}, \eqref{error3}, and the regularity assumption \eqref{regula}, we have
 \begin{equation} 
     \begin{split}
  &  |\sum_{T\in {\cal T}_h}  
  \langle  \nabla e_0\cdot \bn - e_n,{\cal Q}_h\Phi-\Phi \rangle_{\partial T}|\\
 \lesssim &   h^2 \3bar e_h\3bar  \|e_0\|.
    \end{split}
 \end{equation} 

Using the Cauchy-Schwarz inequality, \eqref{error3}, we have
\begin{equation}\label{a2}
    \begin{split}
&       \sum_{T\in {\cal T}_h}  (\Phi-{\cal Q}_h\Phi, e_0)_T\\
=& \sum_{T\in {\cal T}_h}  (\Phi-{\cal Q}_h\Phi, e_0-{\cal Q}_h e_0)_T\\
 \lesssim & \Big(  \sum_{T\in {\cal T}_h}   \|\Phi-{\cal Q}_h\Phi\|^2_{T} \Big)^{\frac{1}{2}}\Big(  \sum_{T\in {\cal T}_h}  \|e_0-{\cal Q}_h e_0\|^2_{ T}\Big)^{\frac{1}{2}}\\
 \lesssim &  h^2\|\Phi\|_{2}\3bar e_h\3bar\\
 \lesssim & h^2\|e_0\|\3bar e_h\3bar.
    \end{split}
\end{equation}

 Substituting \eqref{a1}-\eqref{a2} into \eqref{a3} gives 
 $$
  \|e_0\|^2\lesssim  h^2\3bar e_h\3bar\|e_0\|,
 $$
 which, together with \eqref{erres},  completes the proof of the theorem.
 
\end{proof}

\section{Numerical Experiments}\label{Section:NE}
 In this section, several numerical experiments will be implemented to verify the convergence theory established in previous sections. 

In the first numerical example,  we solve the model problem \eqref{model}-\eqref{model2} on the
unit square domain $\Omega=(0,1)\times(0,1)$.
We let $\Gamma_C=\Gamma_O=\{0\}\times(0,1)\subset \partial \Omega$ and
   $\Gamma_N=\partial \Omega \setminus \Gamma_C$.
We choose the functions in  \eqref{model}--\eqref{model2} as follows: 
\an{\label{f-1}\ad{ f(x,y) &=-(2\pi^2+1) \cos(\pi x)\cos(\pi y)+( \pi^2+1)\cos(\pi y), \\
                q(x,y) &=0, \\
                c_0(x,y)&= 0. } }
            In this case, the optimal control solution is,
   independent of $\alpha$ in \eqref{model2},
\an{\label{s1} u(x,y) &= (1-\cos(\pi x)) \cos(\pi y). }

We apply the WG finite element method \eqref{wg}--\eqref{wg2} to approximate the
  solution \eqref{s1}.
We adopt uniform triangular meshes, as shown in Figure \ref{grid1}.
The computational errors for this problem are listed in Tables \ref{t1}-\ref{t3} when different degrees of polynomial are employed.
Roughly, the optimal orders of convergence are achieved in all cases, verifying the
  theory established in the previous sections.
In all cases, a larger $\alpha$ produces a better solution.

\begin{figure}[ht]
 \begin{center} \setlength\unitlength{1.25pt}
\begin{picture}(260,80)(0,0)
  \def\tr{\begin{picture}(20,20)(0,0)\put(0,0){\line(1,0){20}}\put(0,20){\line(1,0){20}}
          \put(0,0){\line(0,1){20}} \put(20,0){\line(0,1){20}}
   \put(0,20){\line(1,-1){20}}   \end{picture}}
 {\setlength\unitlength{5pt}
 \multiput(0,0)(20,0){1}{\multiput(0,0)(0,20){1}{\tr}}}

  {\setlength\unitlength{2.5pt}
 \multiput(45,0)(20,0){2}{\multiput(0,0)(0,20){2}{\tr}}}

  \multiput(180,0)(20,0){4}{\multiput(0,0)(0,20){4}{\tr}}

 \end{picture}\end{center}
\caption{The first three levels of triangular grids used in computation.}
\label{grid1}
\end{figure}
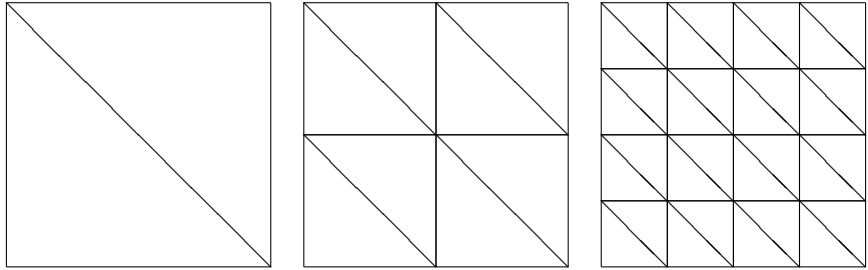
 
\begin{table}[ht]
  \centering \renewcommand{\arraystretch}{1.1}
  \caption{Error profiles on grids shown in Figure \ref{grid1}, for \eqref{s1}. }
\label{t1}
\begin{tabular}{c|cc|cc}
\hline
level & $\3bar Q_h u - u_h \3bar $  &order &  $ \|Q_0 u- u_h  \|_0  $ &order    \\
\hline
 &\multicolumn{4}{c}{by the $P_3$-$P_1$ WG finite element, $\alpha=10^{-6}$. } \\ \hline
 2&    0.260E+00 &  1.9&    0.205E-02 &  3.8 \\
 3&    0.658E-01 &  2.0&    0.169E-03 &  3.6 \\
 4&    0.165E-01 &  2.0&    0.324E-04 &  2.4 \\
\hline
 &\multicolumn{4}{c}{by the $P_3$-$P_1$ WG finite element, $\alpha=10^{0}$. } \\ \hline
 2&    0.260E+00 &  1.9&    0.207E-02 &  3.8 \\
 3&    0.657E-01 &  2.0&    0.129E-03 &  4.0 \\
 4&    0.165E-01 &  2.0&    0.133E-04 &  3.3 \\
\hline
 &\multicolumn{4}{c}{by the $P_3$-$P_1$ WG finite element, $\alpha=10^{6}$. } \\ \hline
 3&    0.358E+00 &  4.0&    0.139E-03 &  4.0 \\
 4&    0.275E-01 &  3.7&    0.867E-05 &  4.0 \\
 5&    0.436E-02 &  2.7&    0.541E-06 &  4.0 \\
 \hline
\end{tabular}%
\end{table}%

\begin{table}[ht]
  \centering \renewcommand{\arraystretch}{1.1}
  \caption{Error profiles on grids shown in Figure \ref{grid1}, for \eqref{s1}. }
\label{t2}
\begin{tabular}{c|cc|cc}
\hline
level & $\3bar Q_h u - u_h \3bar $  &order &  $ \|Q_0 u- u_h  \|_0  $ &order    \\
\hline
 &\multicolumn{4}{c}{by the $P_4$-$P_2$ WG finite element, $\alpha=10^{-6}$. } \\ \hline
 2&    0.309E-01 &  2.9&    0.803E-04 &  4.9 \\
 3&    0.391E-02 &  3.0&    0.274E-05 &  4.9 \\
 4&    0.544E-03 &  2.8&    0.420E-06 &  2.7 \\
\hline
 &\multicolumn{4}{c}{by the $P_4$-$P_2$ WG finite element, $\alpha=10^{0}$. } \\ \hline
 2&    0.309E-01 &  2.9&    0.841E-04 &  4.9 \\
 3&    0.391E-02 &  3.0&    0.368E-05 &  4.5 \\
 4&    0.540E-03 &  2.9&    0.263E-06 &  3.8 \\
\hline
 &\multicolumn{4}{c}{by the $P_4$-$P_2$ WG finite element, $\alpha=10^{6}$. } \\ \hline
 2&    0.700E+00 &  4.9&    0.845E-04 &  4.9 \\
 3&    0.225E-01 &  5.0&    0.274E-05 &  4.9 \\
 4&    0.877E-03 &  4.7&    0.870E-07 &  5.0 \\
 \hline
\end{tabular}%
\end{table}%

\begin{table}[ht]
  \centering \renewcommand{\arraystretch}{1.1}
  \caption{Error profiles on grids shown in Figure \ref{grid1}, for \eqref{s1}. }
\label{t3}
\begin{tabular}{c|cc|cc}
\hline
level & $\3bar Q_h u - u_h \3bar $  &order &  $ \|Q_0 u- u_h  \|_0  $ &order    \\
\hline
 &\multicolumn{4}{c}{by the $P_5$-$P_3$ WG finite element, $\alpha=10^{-6}$. } \\ \hline
 1&    0.443E-01 &  0.0&    0.327E-03 &  0.0 \\
 2&    0.289E-02 &  3.9&    0.511E-05 &  6.0 \\
 3&    0.233E-03 &  3.6&    0.589E-06 &  3.1 \\
\hline
 &\multicolumn{4}{c}{by the $P_5$-$P_3$ WG finite element, $\alpha=10^{0}$. } \\ \hline
 1&    0.445E-01 &  0.0&    0.371E-03 &  0.0 \\
 2&    0.288E-02 &  3.9&    0.496E-05 &  6.2 \\
 3&    0.229E-03 &  3.7&    0.761E-07 &  6.0 \\
\hline
 &\multicolumn{4}{c}{by the $P_5$-$P_3$ WG finite element, $\alpha=10^{6}$. } \\ \hline
 1&    0.446E+01 &  0.0&    0.373E-03 &  0.0 \\
 2&    0.731E-01 &  5.9&    0.497E-05 &  6.2 \\
 3&    0.118E-02 &  6.0&    0.662E-07 &  6.2 \\
 \hline
\end{tabular}%
\end{table}%

In the second numerical example,  we solve \eqref{model}--\eqref{model2} on the
unit square domain $\Omega=(0,1)\times(0,1)$ again.
We let $\Gamma_C=\Gamma_O=\{0\}\times(0,1)\subset \partial \Omega$ and
   $\Gamma_N=\partial \Omega \setminus \Gamma_C$.
We choose the functions in  \eqref{model}--\eqref{model2} as follows:
\a{ f(x,y) &=\Big( \frac{5\pi^2}4 + 1 \Big) \sin\frac{\pi x}2 \cos(\pi y), \\
                q(x,y) &=0, \\
                c_0(x,y)&= 0. }
In this example, the optimal-control solution for each $\alpha$ is unknown.
But when $\alpha\to 0$, we know the solution is
\an{\label{s2} \lim_{\alpha\to 0} u(x,y) &= \cos\frac{\pi(x-1)}2 \cos(\pi y). }
Again, the computation is done on the grids shown in Figure \ref{grid1}.
For small $\alpha$ cases,  we compare the numerical solution with the limit solution
   \eqref{s2} and list the errors in Table \ref{t4}.
For larger $\alpha$,  we do not have an exact solution to compare with.
We simply plot these solutions in Figure \ref{solution2}. 
We can see the surfaces become flat when $\alpha$ is getting big.

\begin{table}[ht]
  \centering \renewcommand{\arraystretch}{1.1}
  \caption{Error profiles on grids shown in Figure \ref{grid1}, for \eqref{s2}. }
\label{t4}
\begin{tabular}{c|cc|cc}
\hline
level & $\3bar Q_h u - u_h \3bar $  &order &  $ \|Q_0 u- u_h  \|_0  $ &order    \\
\hline
 &\multicolumn{4}{c}{by the $P_3$-$P_1$ WG finite element, $\alpha=10^{-8}$. } \\ \hline
 3&    0.214E-01 &  2.0&    0.117E-03 &  2.7 \\
 4&    0.534E-02 &  2.0&    0.258E-04 &  2.2 \\
 5&    0.134E-02 &  2.0&    0.534E-05 &  2.3 \\
\hline
 &\multicolumn{4}{c}{by the $P_4$-$P_2$ WG finite element, $\alpha=10^{-8}$. } \\ \hline
 1&    0.576E-01 &  0.0&    0.537E-03 &  0.0 \\
 2&    0.743E-02 &  3.0&    0.186E-04 &  4.8 \\
 3&    0.940E-03 &  3.0&    0.206E-05 &  3.2 \\
 \hline
\end{tabular}%
\end{table}%

\begin{figure}[ht]
 \begin{center} \setlength\unitlength{1in}
\begin{picture}(4.6,4.5)(0,0)

\put(0,-1.5){\includegraphics[width=4.7in]{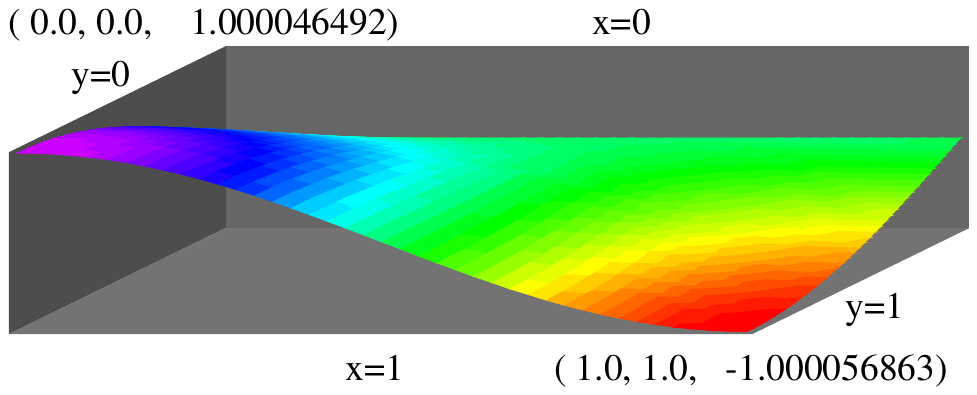}} 
\put(0,-3){\includegraphics[width=4.7in]{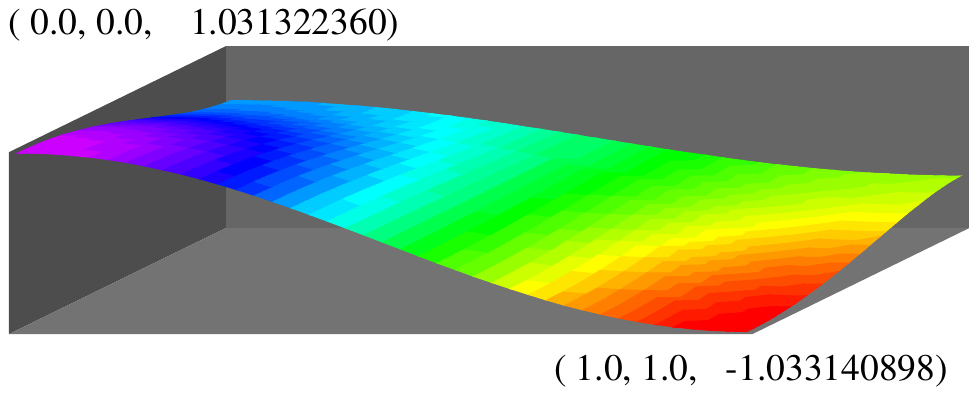}} 
\put(0,-4.5){\includegraphics[width=4.7in]{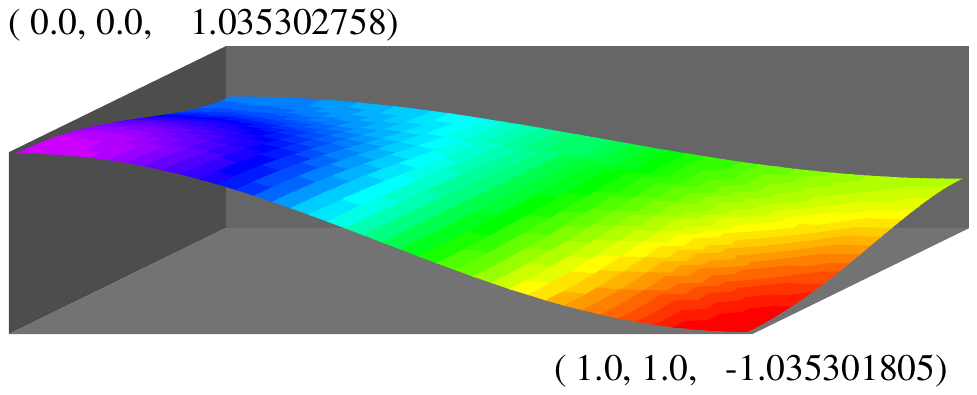}} 

 \end{picture}\end{center} 
\caption{The numerical optimal-solutions in the second example, for
   $\alpha=10^{-4}$ (top), $1$ (middle) and $10^4$ (bottom).}\label{solution2}.
\end{figure}

In the last numerical example, we solve \eqref{model}--\eqref{model2} on the
unit square domain $\Omega=(0,1)\times(0,1)$, with
\a{ \Gamma_C&=\Gamma_O =\{0\}\times(0,1)\subset \partial \Omega, \\
    \Gamma_N&=\partial \Omega \setminus \Gamma_C . }
We choose the functions in  \eqref{model}--\eqref{model2} as follows.  
\a{ f(x,y) &=-16 y^4+32 y^3+176 y^2-192 y+32, \\
                q(x,y) &=0, \\
                c_0(x,y)&= 0. }
We do not know the optimal-control solution for general $\alpha$.
But when $\alpha\to \infty$, we know the solution is
\an{\label{s3} \lim_{\alpha\to 0} u(x,y) &= -16 y^2 (1-y)^2. }
The computation is done on the grids shown in Figure \ref{grid1}.
For a large $\alpha$,  we compare the numerical solution with the limit solution
   \eqref{s3} and list the errors in Table \ref{t5}.
For a small $\alpha$,  we do not have an exact solution to compare with.
We plot these solutions in Figure \ref{solution3}. 
We can see the surfaces become flat in $y$ direction
      when $\alpha$ is getting big.

\begin{table}[ht]
  \centering \renewcommand{\arraystretch}{1.1}
  \caption{Error profiles on grids shown in Figure \ref{grid1}, for \eqref{s3}. }
\label{t5}
\begin{tabular}{c|cc|cc}
\hline
level & $\3bar Q_h u - u_h \3bar $  &order &  $ \|Q_0 u- u_h  \|_0  $ &order    \\
\hline
 &\multicolumn{4}{c}{by the $P_3$-$P_1$ WG finite element, $\alpha=10^{9}$. } \\ \hline
 3&    0.742E-01 &  2.0&    0.929E-04 &  3.9 \\
 4&    0.186E-01 &  2.0&    0.583E-05 &  4.0 \\
 5&    0.464E-02 &  2.0&    0.363E-06 &  4.0 \\
\hline
 &\multicolumn{4}{c}{by the $P_4$-$P_2$ WG finite element, $\alpha=10^{9}$. } \\ \hline
 2&    0.176E-01 &  2.7&    0.254E-03 &  4.6 \\
 3&    0.238E-02 &  2.9&    0.976E-05 &  4.7 \\
 4&    0.307E-03 &  3.0&    0.329E-06 &  4.9 \\
 \hline
\end{tabular}%
\end{table}%

\begin{figure}[ht]
 \begin{center} \setlength\unitlength{1in}
\begin{picture}(4.6,4.5)(0,0)

\put(0,-1.5){\includegraphics[width=4.7in]{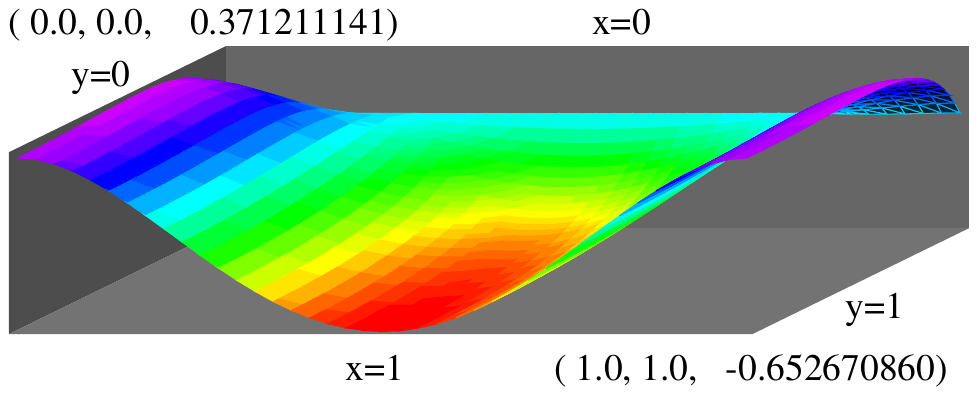}} 
\put(0,-3){\includegraphics[width=4.7in]{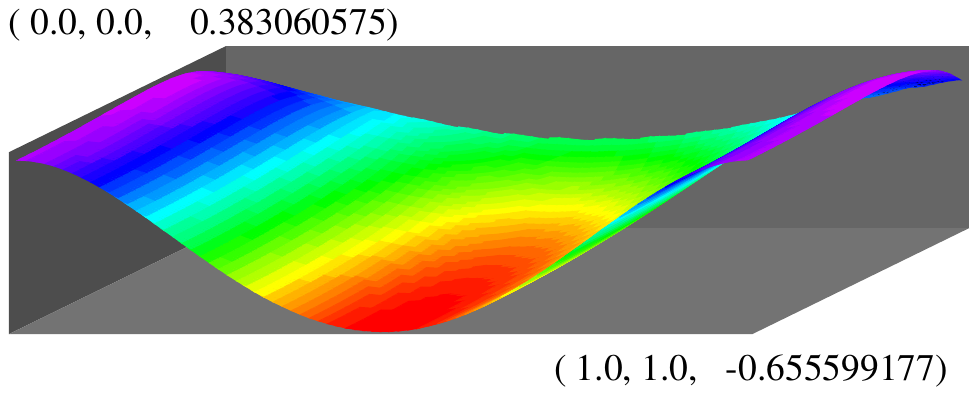}} 
\put(0,-4.5){\includegraphics[width=4.7in]{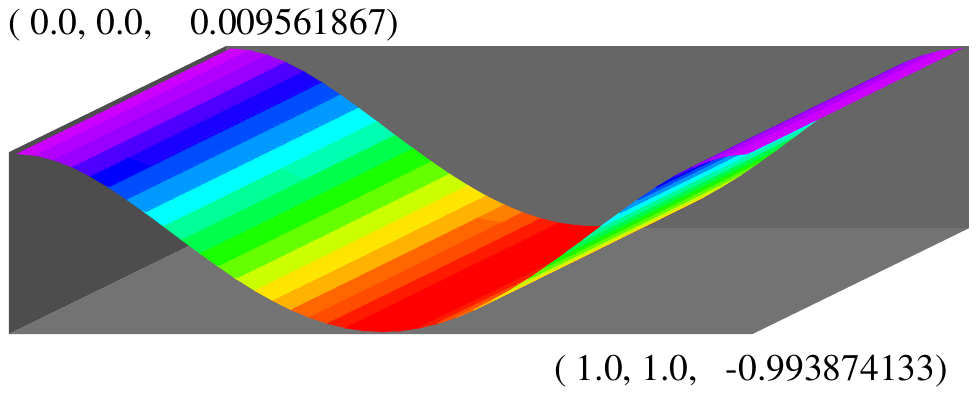}} 

 \end{picture}\end{center} 
\caption{The numerical optimal-solutions in the third example \eqref{s3}, for
   $\alpha=10^{-9}$ (top), $10^{-2}$ (middle) and $10^2$  (bottom).}\label{solution3}.
\end{figure}

\end{document}